\documentclass[12pt,letterpaper]{amsart}
\usepackage{amsmath, amsthm, amssymb, xspace, mathrsfs}
\usepackage[tmargin=1.4in,bmargin=1.2in,rmargin=1.4in,lmargin=1.4in]{geometry}
\usepackage[breaklinks=true]{hyperref}%To get clickable links to papers
\usepackage{amsmath,amscd}
\usepackage{tikz-cd}

\theoremstyle{plain}
\newtheorem{theorem}{Theorem}[section]
\newtheorem{lemma}{Lemma}[section]
\newtheorem{prop}{Proposition}[section]

\newtheorem{conjecture}{Conjecture}[section]

\theoremstyle{definition}
\newtheorem{defin}{Definition}[section]
\newtheorem{remark}{Remark}[section]
\newtheorem*{acknowledgement}{Acknowledgements}

\newcommand{\comment}[1]{}

\begin{document}
\title{On the Dirichlet problem in the plane with polynomial data}
\author{Akaki Tikaradze}

\email{ Akaki.Tikaradze@utoledo.edu}
\address{University of Toledo, Department of Mathematics \& Statistics, 
Toledo, OH 43606, USA}

\begin{abstract}
Let $\Omega\subset\mathbb{C}$ be a bounded domain such that 
there exists an algebraic harmonic function of degree two vanishing on the boundary of $\Omega.$
Then we show that the Khavinson-Shapiro conjecture holds for $\Omega:$ if
the Dirichlet problem on $\Omega$ with all polynomial boundary data
have polynomial solutions, then $\Omega$ must be an ellipse. We also prove that if there exists a rational function with a singularity in $\Omega$, such that the Dirichlet problem
for its restriction on $\partial\Omega$ along with all polynomial functions have rational solutions, then $\Omega$ must be a disc.
This generalizes a well-known result by Bell, Ebenfelt, Khavinson, and Shapiro. Our proofs are purely algebraic.
\end{abstract}

\maketitle

\section{Introduction}

Let $\Omega$ be a domain in $\mathbb{C}.$ Let $v\in C(\partial\Omega).$ Recall that solving the Dirichlet problem
on $\Omega$ with the boundary data $v$ amounts to finding $u\in C^2(\Omega)\cap C(\overline{\Omega}$), such that
$u$ is harmonic on $\Omega$ and $u|_{\partial\Omega}=v.$ Domains for which the Dirichlet problem with algebraic data admit algebraic
solutions have been of considerable interest for some time now.
It is of particular interest to consider domains for which the Dirichlet problem with any polynomial data has a polynomial solution.
This is known to be the case when $\Omega$ is an ellipse by a short and elegant argument of Fischer, which led Khavinson and Shapiro \cite{KS}
to make the following conjecture (they made the conjecture about bounded domains in $\mathbb{R}^n$, we restrict ourselves to $n=2$ case).

\begin{conjecture}[Khavinson-Shapiro conjecture]
Let $\Omega\subset\mathbb{C}$	 be a bounded domain
whose boundary consists of
finitely many non-intersecting Jordan curves. Suppose that 
the solution of the Dirichlet problem on $\Omega$ with every polynomial data
is again polynomial. Then $\Omega$ must be an ellipse.

\end{conjecture}
At present, the Khavinson-Shapiro conjecture is still wide open.
See \cite{CS}, \cite{R1}, \cite{R2},  \cite{LR} for some partial results. 

In this paper we show the following.

\begin{theorem}\label{KH-main}

Let $\Omega\subset\mathbb{C}$ be a bounded domain so that $\partial\Omega$ has no isolated points.
Assume that there exist rational holomorphic functions $f, g\in \mathbb{C}(z)$ and a rational harmonic
function $r$, not all 0, such that 
$$r^4-4(f+\bar{g})r^2+(f-\bar{g})^2=0|_{\partial{\Omega}}.$$
Then the Khavinson-Shapiro conjecture holds for $\Omega.$

\end{theorem}

The existence of $r, f, g$ as above is (essentially) equivalent to the existence of a a harmonic
algebraic function of degree 2 that vanishes on $\partial\Omega.$

We also consider the following algebraic generalization of the Khavinson-Shapiro conjecture for arbitrary fields.

\begin{conjecture}\label{KS}
Let $F$ be a field.
Let $f\in F[z, w]$ be a polynomial such that for any $\phi\in F[z, w]$
there exists $\phi'\in F[z, w],$ so that $\phi-f\phi'\in F[z]+F[w]$. Then either the total degree of $f$ is at most 2, or 
$f$ is linear in $z$ or $w.$

\end{conjecture}

This conjecture over $\mathbb{C}$ can be seen as a complexified version of  [\cite{CS}, Conjecture4].
As proved by Render \cite{R2},  [\cite{CS}, Conjecture4] implies the Khavinson-Shapiro
conjecture (assuming that the boundary of the domain has no isolated points).

We show that it is enough to prove the above conjecture for number fields to conclude that it holds for any characteristic 0 field.
We also show that proving this conjecture over prime finite fields $\mathbb{F}_p$ for $p\gg 0$ implies its validity for all fields.

We prove the following

\begin{theorem}\label{KS-theorem}
Let $h_1, h_2\in F[t]$ be irreducible polynomials over a field $F$ and $\deg(h_1)\leq \deg(h_2).$
Assume that either  $\deg(h_1)$ does not divide $\deg(h_2),$\\ or $\deg(h_1)=\deg(h_2)$
and one of $F[t]/(h_1), F[t]/(h_2)$ is not a splitting field over $F.$ Then for any $\phi, \psi\in F[z, w],$
Conjecture \ref{KS} holds for $h_1(z)\phi+h_2(w)\psi.$ In particular, 
let $a, b$ be square-free  odd integers $|a|, |b|>1$, and assume that $n$ is not a power of $2.$ Then Conjecture \ref{KS} holds for
$(z^n-a)\phi+(w^n-b)\psi$ for any $\phi, \psi\in \mathbb{Q}(i)[z, w].$
\end{theorem}

After polynomials, the next natural choice to consider for the solvability of the Dirichlet problem
is the class of  rational functions. It is well-known and  easy to check that if
$\Omega$ is a disk, then for any $v\in C(\partial\Omega)$ such that $v$ is a restriction of a rational function, the Dirichlet
problem with the data $v$ has a rational solution. Thus, in the spirit of the Khavinson-Shapiro conjecture, it is natural to ask
whether this property characterizes disks. This question was settled affirmatively
by the following result of Bell, Ebenfelt, Khavinson, and Shapiro.

\begin{theorem}[\cite{BEKS}, Theorem 2]
Let $\Omega\subset\mathbb{C}$	 be a bounded domain
whose boundary consists of
finitely many non-intersecting Jordan curves and let $a\in\Omega.$ Suppose that 
the solution of the Dirichlet problem on $\Omega$ with every polynomial data
and $\frac{1}{z-a}|_{\partial\Omega}$ is rational. Then $\Omega$ is a disk.
\end{theorem}

The proof given in \cite{BEKS} relies on some rather nontrivial complex analysis.
Under the assumption that $\Omega$ is simply connected, then \cite{BEKS} shows a stronger result: if the Dirichlet problem on $\Omega$
for $z\bar{z}|_{\partial\Omega}, z^2\bar{z}|_{\partial\Omega}, z^3|_{\partial\Omega}$ ,and $\frac{1}{z-a}|_{\partial\Omega}$ (for some $a\in\Omega$) admits
rational solutions, then $\Omega$ is a disk.

We prove the following generalization of the above theorem to arbitrary (possibly unbounded) domains, 
and what is more important replacing
the rational function $\frac{1}{z-a}$ with an arbitrary rational function that has a singularity in the domain.

Recall that the upper half plane $\Omega=\lbrace z| \text{Im}(z)>0\rbrace$ also has the property that
the Dirichlet problem with any rational data admits a rational solution. Our result shows that a disk (or its complement)
and a right half plane are the only such domains (up to removing finitely many points).

\begin{theorem}\label{main}
Let $\Omega\subset \mathbb{C}$ be a domain, such
that the Dirichlet problem on $\Omega$ with any polynomial data has a rational solution.
Assume moreover that  there exists a rational harmonic function $\phi$ with a singularity in $\Omega,$ such that 
the Dirichlet problem on $\Omega$ with the data
$\phi|_{\partial\Omega}$
admits a rational solution. Then $\partial\Omega$ except for finitely many points is either a subset of a  circle or a line.
\end{theorem}

  In fact, we prove much more general algebraic statement. To state it, recall that given a holomorphic function
  $u$ (defined on an open subset in $\mathbb{C}$), we say that $u$ is an algebraic function of degree (at most) $n$
  if there exist rational functions $r_i(z), 0\leq i\leq n$ (not all of them zero) such that $\sum_{i=0}^nr_i(z)u(z)^i=0.$
  
%\begin{theorem}\label{algebraic}

%Let $\Gamma\subset \mathbb{C}$ be an infinite set on which a nontrivial rational harmonic function vanishes.
%Assume that for any polynomial $u\in\mathbb{C}[z, \bar{z}]$, there exists a rational harmonic function $v$, such that $u=v|_{\Gamma}.$
%Then $\Gamma$ except for finitely many points  must be a subset of either a circle or a line.
%\end{theorem}

\begin{theorem}\label{algebraic}

Let $\Gamma\subset \mathbb{C}$ be an infinite set.
 Assume that for any polynomial $u\in\mathbb{C}[z, \bar{z}]$, there exists a rational harmonic function $v$, such that $u=v|_{\Gamma}.$
Suppose moreover that there exist algebraic functions $f, g$ of degree at most $n$, so that $f-\bar{g}$ vanishes on $\Gamma.$
Then there exists 
a nonzero polynomial $\phi\in \mathbb{C}[z, \bar{z}]$ such that the degree of $\phi$ in $z$ and in $\bar{z}$ is at most
$n$ and $\phi$ vanishes on $\Gamma$ except for finitely many points.
\end{theorem}

Our proof relies solely on some elementary algebraic considerations.
%As a direct corollary, we obtain the aforementioned generalization of Theorem \ref{main}

\begin{remark}

Under the assumption that the Dirichlet problem for restrictions on $\partial\Omega$ of all polynomials and $\frac{1}{z-a}$ (for some $a\in\Omega$)
has an algebraic solution, it was proved in [\cite{BEKS}, Theorem 3] that $\Omega$ must be simply connected with an algebraic Riemann mapping.

\end{remark}

\section{Proofs of theorems \ref{main},  \ref{algebraic}}

At first, remark that the existence of a rational harmonic function $\phi$ with a singularity in $\Omega,$ such that the Dirichlet problem on $\Omega$
with the data $\phi|_{\partial\Omega}$ has a rational solution implies  the existence of a nonzero rational harmonic function vanishing on $\partial{\Omega}.$
Indeed, if $u$ is a harmonic function on $\Omega$ such that  $u=\phi|_{\partial\Omega}$, then $u-\phi$ is a nonzero harmonic rational
function vanishing on the boundary of $\Omega.$
In view of this, Theorem \ref{main} follows immediately from Theorem \ref{algebraic} by putting $n=1.$

The proof of Theorem \ref{algebraic} is based on the following simple algebraic observation.

\begin{prop}\label{key}

Let $A$ be a finite dimensional algebra over a field $F.$
Let $ F_1, F_2$ be subfields of $A$ containing $F.$
If $F_1+F_2=A$ then $A=F_1$ or $A=F_2.$
\end{prop}
\begin{proof}
Since $\dim_F(F_1\cap F_2)\geq 1,$  then
 $$\dim_F(A)\leq \dim_FF_1+\dim_FF_2-1.$$ On the other hand,
as $A$ can be viewed as a vector space over $F_1$ and $F_2,$ we conclude that $\dim_FF_1, \dim_FF_2$ must divide $\dim_F(A).$
Therefore $$\dim_F(A)=\text{max}(\dim_FF_1, \dim_FF_2).$$
Hence, $A=F_1$ or $A=F_2.$

\end{proof}

We also recall the following very simple fact. We include the proof for the reader's convenience 

\begin{lemma}\label{trivial}
Let $u$ be an algebraic function of degree $n.$ Then $z$ is algebraic over $\mathbb{C}(u)$, and any
element in $\mathbb{C}(u, z)$ is an algebraic function of degree at most $n.$

\end{lemma}
\begin{proof}
As $u$ is algebraic over $\mathbb{C}(z),$ then $u, z$ are algebraically dependent over $\mathbb{C}$, hence $z$ must be
algebraic over $\mathbb{C}(u).$
Since $u$ is a root of a degree $n$ polynomial over $\mathbb{C}(z)$, it follows that the degree of the field extension
$\mathbb{C}(u, z)/\mathbb{C}(z)$ is at most $n.$ Thus, every element in $\mathbb{C}(u, z)$ is an algebraic function
of degree at most $n.$ 
%We have that for rational functions $r_i, 0\leq i\leq n$ (not all of them 0) $\sum_{i=0}^nr_i(z)u^i=0.$

\end{proof}

\begin{proof}[Proof of Theorem \ref{algebraic}] 
By the assumption on $\Gamma$,  there exist nonconstant holomorphic functions
$f, g$ defined on a neighborhood of $\Gamma$, such that $f(z)=\overline{g(z)}|_{\Gamma},$ 
and $f, g$ are algebraic of degree at most $n.$
Let $\mathcal{O}$ denote the $\mathbb{C}$-algebra of functions defined on $\Gamma$  except possibly for finitely many
points. Note that the field of algebraic functions defined on a neighborhood of $\Gamma$ maps into
$\mathcal{O}$ by the restriction homomorphism.

Denote by $h$ the image of $f$ (same as the image of $\bar{g}$) in $\mathcal{O}.$
Let $A$ denote the subalgebra of $\mathcal{O}$ generated by images of
 $z, \bar{z}$ over $\mathbb{C}(h)=F.$
Thus, $F$  is a field and  $A$ is an algebra over  $F.$
 We identify $z, \bar{z}$ with their images in $\mathcal{O}.$
Next we claim that 
$$A=F[z, \bar{z}]=F[z]+F[\bar{z}].$$ Indeed, since $z,\bar{z}$ are algebraic over $F$ by Lemma \ref{trivial}, we have 
$$F(z)=F[z],\quad F(\bar{z})=F[\bar{z}].$$
So, $A=F[z, \bar{z}].$ In particular, $\dim_FA<\infty.$ We need to show that 
$$z^n\bar{z}^m\in F[z]+F[\bar{z}], \quad n, m\geq 0.$$
This follows from the assumption that there exists a rational harmonic function $u\in \mathbb{C}(z)+\mathbb{C}(\bar{z}),$
so that $u=z^n\bar{z}^m|_{\Gamma}.$ Hence
 $$z^n\bar{z}^m\in \mathbb{C}(z)+\mathbb{C}(\bar{z})\subset F[z]+F[\bar{z}].$$

\noindent Put $F_1=F[z], F_2=F[\bar{z}].$ Thus $F_1, F_2$ are subfields of $A$ containing $F$ and $A=F_1+F_2.$
Therefore, we can apply Proposition \ref{key} to conclude that $A=F_1$ or $A=F_2.$
Assume without loss of generality that $\bar{z}\in \mathbb{C}(h)[z].$

To summarize, we have proved that there exists  $u\in \mathbb{C}(f)[z]$ such that $\bar{z}=u|_{\Gamma}.$
It follows from Lemma \ref{trivial} that the degree of $u$ over $\mathbb{C}(z)$ is at most $n.$ Therefore,  there exist polynomials
$p_i\in \mathbb{C}[z]$, such that for any $z\in\Gamma$  (except for finitely many points)
$$\sum _{i=0}^np_i(z)\bar{z}^i=0.$$
Therefore,  $$\sum _{i=0}^n\overline{p_i(z)}{z}^i=0|_{\Gamma}.$$
Let $h$ be the greatest common divisor of $\sum _{i=0}^np_i(z)\bar{z}^i$ and $\sum _{i=0}^n\overline{p_i(z)}{z}^i$
in $\mathbb{C}[z,\bar{z}].$
Then the degree of $h$ in $z, \bar{z}$ is at most $n$ and $h$ vanishes on $\Gamma$ (except for finitely many points), as
desired.

%To summarize, we have shown that  there exists a rational function $r(z)\in\mathbb{C}(z)$, such that $\bar{z}=r(z)|{\Gamma}.$
%Then it well-known and easy to prove that if the zero set of $\bar{z}-r(z)$ is infinite for a rational function $r(z)$, then $r(z)$ must be a M\"{o}bius transform and
%hence $\Gamma$ is a subset of a circle or a line. Indeed,  put $r^*(z)=\overline{r(\bar{z})}.$
%It follow that for $z\in\Gamma$, we have that $z={r}^*(r(z)).$
%Since $\Gamma$ is infinite, it follows that $r^{-1}=r^*$ is rational. Since the only invertible rational transformations of $\mathbb{C}$
%are M\"{o}bius functions, we conclude that r must be a M\"{o}bius transform. Therefore, $\Gamma$ is contained in a circle
%or a line.

%To summarize, we have proved that there exist coprime polynomials $f(z), g(z)\in \mathbb{C}[z],$ such that $\bar{z}-\frac{f(z)}{g(z)}$ vanishes on $\Gamma.$
%Then it well-known and easy to prove that if the zero set of \bar{z}-r(z) is infinite for a rational function r(z), then r(z) must be a mobius transform (hence
%Thus, $g(z)\bar{z}-f(z), \bar{g}z-\bar{f}$ are polynomials in $\mathbb{C}[z, \bar{z}]$ vanishing on $\Gamma.$
%Since $\Gamma$ is infinite, they must have a common factor $h$ that vanishes on $\Gamma$ except for finitely many points:
%$$h|g(z)\bar{z}-f(z), h|\bar{g}z-\bar{f}.$$
%Hence $h$ must be of degree one in both $z, \bar{z}.$
%Thus, $h$ must be of the form $z\bar{z}+b\bar{z}+c$ or $az+b\bar{z}+c.$
%Hence, $\Gamma$ minus a finite set is a subset of a circle or a line.
\end{proof}

\section{Results on the Khavinson-Shapiro conjecture}

It will be convenient to introduce the following terminology

\begin{defin}\label{KS-property}
Let $F$ be a field.
Let $f\in F[z, w].$ We say that $f$ is a  Khavinson-Shapiro polynomial (KS-polynomial for short)
if for any $\phi\in \mathbb{F}[z, w] $ there exist $g_1(z)\in F[z], g_2(w)\in F[w],$ such that
$\phi-g_1(z)-g_2(w)$ is a multiple of $\phi.$

\end{defin}

Throughout polynomials of the form $f(z)+g(w)\in F[z, w]$ are referred to as harmonic polynomials in $F[z, w].$

We recall the following classical result of Fischer. We include the usual proof
for the reader's convenience as it is stated over any field (including ones with positive characteristic).

\begin{lemma}[Fischer]

Let $f\in F[z, w]$ be of degree 2. If $f$ does not divide any nonzero harmonic polynomial, then $f$ is a $KS$-polynomial.

\end{lemma}  

\begin{proof}
Denote by $D: F[z, w]\to F[z, w]$ a linear map defined by $D(z^nw^m)=z^{n-1}w^{m-1}$ if $n, m>0$ and 0 otherwise.
Then $\ker(D)$ consists of harmonic polynomials and $D$ is onto.
So, by the assumption on $f$ the linear map 
$$G: F[z, w]\to F[z, w], G(\phi)=D(f\phi)$$
 is injective. On the other hand,  for any $m\geq 0,$ $G$ preserves the finite dimensional
subspace of all polynomials of degree at most $m.$  Therefore, $G$ is bijective. 
Given $\phi\in F[z, w],$ 
let $\phi_1$  be such that $D(\phi)=D(f\phi_1).$
So, $\phi-f\phi_1$ is harmonic and we are done.

\end{proof}

\begin{lemma}

Let $\phi\in F[z, w]$ be linear in $z.$ Then $\phi$ is a KS-polynomial if and only if $\phi=f(w)z-g(w)$ with linear $f.$

\end{lemma}
\begin{proof}
By the assumption $f(w)z-g(w)$ divides $zw+h(z)+h_1(w)$ for some $h, h_1.$
Comparing the leading terms in $z,$ $f(w)$ must divide a linear polynomial in $w,$ hence the result.

\end{proof}
The proof of the next result is essentially identical to that of Theorem \ref{algebraic}, so it is omitted.

\begin{theorem}
Let $f\in F[z, w]$ be an irreducible polynomial nonlinear in both $z, w$ and be a divisor of a nonzero polynomial of the form
$\phi(z)\psi(w)+\phi_1(z)\psi_1(w).$ Then $f$ is not a KS-polynomial.

\end{theorem}

%\begin{proof}
%Let $A=F[z,w]/(f).$ Then $A=F[z]+F[w]$. By proposition \ref{key}, we may assume without loss of generality that
%$w\in F[z] $(in $A$). It follows that $f$ divides a polynomial of the form $g_1(z)w-g_2(z).$ Hence $f$ is linear in $w$ and we are done.

%\end{proof}

%\begin{prop}
%Conjecture \ref{KS} implies the Khavinson-Shapiro conjecture.

%\end{prop}
%\begin{proof}
%Let $I$ denote the ideal in $\mathbb{C}[z, w]$ (where $w=\bar{z}$) vanishing on $\partial\Gamma.$
%Let $f$ be a polynomial of the smallest degree such that $I\subset f.$ Therefore, $f$ is a KH-polynomial, hence by the assumption
%$\deg(f)\leq 2.$ Now we are done by \cite{CS}.
%\end{proof}

The following result shows that in order to resolve Conjecture \ref{KS},  it is enough to treat the case of $F$ being a finite extension of $\mathbb{Q}$ or a finite field.
Given a polynomial $f\in F[z, w],$ denote by $Supp(f)$ the set of pairs $(n, m)$ for which $z^{n}w^m$ has a nonzero coefficient in $f.$
\begin{prop}

Let $f\in F[z, w]$ be a KS-polynomial. Then there exists infinitely many primes $p$ and polynomials $g\in \mathbb{F}_p[z, w]$
so that $Supp(g)=Supp(f)$ and $g$ is a KS-polynomial. If in addition $F$ has characteristic 0, then $F$ can be replaced by a number field.

\end{prop}
\begin{proof}
For $n, m\geq 1$, let $g_{n,m}\in F[z, w]$ be such that $z^nw^m+fg_{n,m}\in F[z]+F[w].$ Let $deg(g_{n, m})=a_{n,m}.$
Put $b_{nm}=a_{n,m}+deg(f)+n+m.$ Given a natural number $l,$ we denote by $F[z, w]_l$ the $F$-space of polynomials of the total degree $\leq l.$
 Let $\phi_{n,m}: F[z, w]_{a_{nm}}\to zwF[z, w]_{b_{nm}}$
be the projection of the $F$-linear map $h\to z^{nm}+ fh$ (the projection discarding all monomials not containing both $z,w$).
Thus $\phi_{n,m}(g_{n,m})=0,$ so $\phi_{n, m}$ is not injective. Viewing $\phi_{n,m}$ as a matrix with respect to
the monomial basis, its coefficients
clearly belong to integer span of coefficients of $f$ and 1. Therefore, the fact that this matrix has rank less than dim $F[z, w]_{a_{nm}}=d(n,m)$
is equivalent to the vanishing of determinants of all $d(n,m)$ by $d(n,m)$ minors. Each such minor is a polynomial in 
$\mathbb{Z}[f_{i,j}],$ where $f_{i,j}$ is the coefficient of $z^{i}w^j$ in $f.$ 

To summarize, for each triple $n, m, k,$ we have constructed a family of polynomials $S_{n,m,k}\subset \mathbb{Z}[f_{ij}]$ with
the following property.
There exists $g\in F[z, w]$ of degree at most $k$ so that $gf+z^nw^m\in F[z]+F[w]$ if and only if all polynomials in $S_{n, m, k}$ vanish on coefficients of $f.$
Let $S$ denote the ring generated over $\mathbb{Z}$ by all nonzero coefficients of $f$ and their inverses. It is known that
for infinitely many primes $p$, there exists ring homomorphisms $S\to \mathbb{F}_p$ \cite{VWW}. Denoting the image of $f$ by $\bar{f} \in \mathbb{F}_p[z, w],$
we obtain that $\bar{f}$ is a KS-polynomial, moreover $Supp(\bar{f})=Supp(f).$ 
Similarly, we may replace $F$ by a number field.

\end{proof}

The next observation is crucial.
\begin{lemma}\label{crucial}

Let $f\in F[z, w]$ be a KS-polynomial. Let $(a, b)\in \bar{F}^2$ be a zero of $f.$
Then either $a\in F[b],$ or $b\in F[a].$

\end{lemma}
\begin{proof}
Let $A$ denote $F[a, b].$ Thus, $A$ is a finite extension of $F$. Since $f$ is a KS-polynomial, we may conclude that $A=F[a]+F[b].$
Since both $F[a], F[b]$ are finite field extensions of $F$, we are done by Proposition \ref{key}.

\end{proof}

Now we can easily prove Theorem \ref{KS-theorem} as follows. In view of Lemma \ref{crucial}, it suffices to
find $a, b\in \bar{F}$ roots of $h_1, h_2$ respectively so that
$F[a], F[b]$ do not contain each other.
 If $\deg(h_1)$ does not divide $\deg(h_2),$ since 
 $$\dim_FF[a]=\deg(h_1), \dim_FF[b]=\deg(h_2),$$ 
 we conclude that $F[a]$ is not a subfield of $F[b].$
Suppose that  $\deg(h_1)=\deg(h_2)$ and $F[t]/(h_1)$ is not a splitting field over $F.$
Then there exist $a, a'\in\bar{F}$ roots of $h_1,$ such that $a'\notin F(a).$ Then for any $b\in\bar{F}$ that
is a root of $h_2$, we have that $F(b)\neq F(a)$ or $F(b)\neq F(a')$. Let $F(a)\neq F(b).$ Since 
$$[F(a):F]=\deg(h_1)=\deg(h_2)=[F(b):F],$$
we conclude that $F(a), F(b)$ do not contain each other, as desired.
Finally, let $a, b\in \mathbb{Z}$ be square-free odd integers, $|a|, |b|>1.$
 Then it follows from the Eisenstein's irreducibily criterion that $z^n-a, w^n-b$ are irreducible over $\mathbb{Q}(i).$
 The splitting field of $z^n-a$ contains $\mathbb{Q}(i)(\xi_n)$, where $\xi_n$ is a primitive $n$-th root of unity.
 So,  $\mathbb{Q}(i)(\xi_n)=\mathbb{Q}(\xi_m)$, where $m=lcm(4,n)$. So 
 $$[\mathbb{Q}(i)(\xi_n): \mathbb{Q}(i)]=\phi(m)/2.$$
It is clear that $\phi(m)/2$ cannot divide $n$, so $\mathbb{Q}(i)(\xi_n)$ is not contained
in $\mathbb{Q}(i)(a^{\frac{1}{n}})$ and we are done.
%\begin{theorem}
%Let $f(z)\in\mathbb{Q}(i)[z]$ be an irreducible polynomial, such that the degree of its splittng field over $Q(i)$ is larger than 
%$deg(f).$ Let $g(z, \bar{z})\in\mathbb{Q}(i)[z, \bar{z}].$ Then the Khavinson-Shapiro conjecture holds for any domain $\Omega$
%such that $Re(fh)|_{\partial\Omega}.$

%\end{theorem}

%\noindent\textbf{Acknowledgement:} The paper owes its existence to  Steven Bell's Shoemaker lecture series
%.given at University of Toledo in 2015. I am very grateful to Dima Khavinson for many useful comments. 

We can easily derive Theorem \ref{KH-main} from Theorem \ref{algebraic}. Since 
$$r^4-4(f+\bar{g})r^2+(f-\bar{g})^2=(r+\sqrt{f}+\sqrt{\bar{g}})(r+\sqrt{f}-\sqrt{\bar{g}})(r-\sqrt{f}+\sqrt{\bar{g}})(r-\sqrt{f}-\sqrt{\bar{g}})$$
it follows that $\partial\Omega$ can be written as a finite union $\cup_{i=1}^m\Gamma_i,$ so that for each $\Gamma_i$ there exists
a harmonic algebraic function of degree 2 vanishing on it. So, by Theorem \ref{algebraic}, there exists a polynomial $f_i\in\mathbb{C}[z, \bar{z}]$ 
of order at most
2 in both $z, \bar{z}$ vanishing on $\Gamma.$ In fact, such a polynomial must have total degree at most 2. Indeed, since
$$z\bar{z}+h(z)+\overline{h(z)}|_{\partial\Gamma}=0$$ for some $h\in\mathbb{C}[z],$  we can replace $f_i$ by gcd of $f_i$ and $z\bar{z}+h(z)+\overline{h(z)}$
which will have the total degree at most 2. Thus, $\partial\Omega$ is a union of finitely many curves each of them being either a line, ellipse, parabola or hyperbola. If an infinite subset of $\partial\Omega$ is a zero of an irreducible quadratic function $\phi$, then 
$\phi||z|^2+h$ for some real harmonic polynomial $h.$ Therefore $h$ is at most quadratic, implying that $\partial\Omega$ is contained in the
zero set of an irreducible quadratic polynomial.
So, $\Omega$ must be an ellipse. Thus, the only case left to consider is when $\Omega$ a polygon.
Then we are done by \cite{R2}

\begin{acknowledgement}
The paper owes its existence to  Steven Bell's Shoemaker lecture series
given at University of Toledo in 2015. I am very grateful to Dima Khavinson for many useful comments.
\end{acknowledgement}

\end{document}